\documentclass[fleqn,a4paper,11pt]{article}
\usepackage{amsmath}
\usepackage{dsfont}
\usepackage{a4wide}
\usepackage{color}
\usepackage{amssymb}
\usepackage[mathscr]{eucal}
\usepackage{enumerate}
\usepackage{longtable}
\usepackage{hhline}

\newtheorem{definition}{Definition}[section]
\newtheorem{secsatz}[definition]{Proposition}
\newtheorem{theorem}[definition]{Theorem}
\newtheorem{remark}[definition]{Remark}
\newtheorem{notation}[definition]{Notation}
\newtheorem{conclusion}[definition]{Conclusion}

\newcommand{\pn}[5]{\big(({#1}^{-#2},#4)^{-#3}, #5 \big)}

\newenvironment{proof}{
    \medskip

    \textbf{Proof:}}{
    \nopagebreak
    \vspace{-2.0ex}
    \begin{flushright}
        \tiny $\blacksquare$
    \end{flushright}
    \smallskip

}

\begin{document}


\title{$\varepsilon$-Equilibria in Quitting Games -- Basics}
\author{Katharina Fischer\footnote{This article appeared under the former family name Heimann in the conference proceedings of the "International Symposium on Dynamic Games and Applications", Wro{\l}aw 2008.  }}
\date{\today}

\maketitle

\begin{abstract}
Quitting games are one of the simplest stochastic games in which at any stage each player has only two possible actions, \emph{continue} and
\emph{quit}. The game ends as soon as at least one player chooses to quit. The players then receive a payoff, which depends on the set of players that did choose to quit.
If the game never ends, the payoff to each player is zero.
\medskip

Examples of quitting games were studied first by Flesh, Thuijsman and Vrieze in 1997 (\cite{Flesh}). Solan  1999 (\cite{Sol}) proved that all three-player quitting games have approximate equilibria.
In the paper \textit{Quitting Games} Solan and Vieille 2001 (\cite{SolVie-QG}) proved the existence of subgame-perfect approximate equilibria under some restrictions on the payoff function.
Furthermore Solan and Vieille studied in \cite{SolVie-Bsp} a four-player quitting game example in which the simplest equilibrium strategy is periodic with period two.
In \textit{The structure of non-zero-sum stochastic games} (\cite{Simon}) Simon showed under which properties quitting games have approximate equilibria among other things by generalization of the solution-idea from Solan an Vieille.
\medskip

This paper gives a short introduction into the topic quitting games and tries to illustrate several properties with examples. First the mathematical model of a quitting game is presented. After the definition of the strategy and strategy profile the corresponding probability space and the underlying stochastic process are stated. This leads to the expected payoff and the definition of some equilibria.

For a better analysis of quitting games the so called one-step game is introduced in the second part of this paper. Important properties of strategy profiles in one-step games are posted and proved.

In the third section an important theorem from Solan and Vieille (cf. \cite{SolVie-QG}) is cited, in which the existence of approximate equilibria, under some assumptions to the payoff function, is postulated. It's proof is divided into three parts, however this paper concentrates only on the first one. In the referred literature only a few steps of the proof are denoted. It is the aim to show the proof at length under usage of the then known results.
\end{abstract}
\bigskip

\section{The model}

A quitting game is a sequential $N$-player ($N \in \mathds{N}$) game and played as follows. In every game turn each player has only two possible actions \textit{continue} and \textit{quit}. The game ends as soon as at least one of the $N$-players chooses to quit. We denote $S$ (quitting coalition) as the subset of the players who choose to quit. If $S = \emptyset$ the players receive no payoff and the game continues to the next stage. If $S \ne \emptyset$ each player $n \in \{1, \ldots, N \}$ receives the payoff $r^n_S \in \mathds{R}$ and the game terminates.

\begin{definition}[Quitting Game]
A quitting game is a tuple
\begin{equation}
G = (\mathcal{N}, (r_S)_{\emptyset \subseteq S \subseteq \mathcal{N}})
\end{equation}
where
\begin{itemize}
    \item[--]   $\mathcal{N} = \{1, \ldots, N\} \subset \mathds{N}$ is a finite set of
                players, $N \in \mathds{N}$,
    \item[--]   $ S \in \mathcal{P}( \mathcal{N})$ denotes the quitting coalition and
    \item[--]   $(r_S)_{S \in \mathcal{P}(\mathcal{N})} \in \mathds{R}^N$ is a sequence
                of payoff-vectors to the players under the quitting coalition $S$
                with $r_{\emptyset} = \underline{0}$ ($\underline{0}:=(0,\ldots, 0)^T \in \mathds{R}^N$) and $r_S = (r_S^1, \ldots, r_S^N)^T$.
\end{itemize}
\end{definition}

\begin{remark}\label{Bem1}
A quitting game is a special case of a (stochastic) game, where transition probabilities are even deterministic. For comparison (cf. e.g. \cite{Mojca}):
\begin{itemize}
    \item[--]   The state space is given by $Z := \{S \ | \ \emptyset \subseteq S \subseteq \mathcal{N} \} = \mathcal{P}(\mathcal{N})$.
    \item[--]   The action space is given by $A := \{0, 1\}^N$, where 0 stands for
                \emph{continue} and 1 for \emph{quit}. We denote $a_S=(a^1_S, \ldots, a_S^N)^T$ as element of $A$ with
                \begin{equation}
                a_S^n := \begin{cases}
                    0 & \mbox{ for $n \in \mathcal{N}\setminus S$} \\
                    1 & \mbox{ for $n \in S$}
                \end{cases} \quad \forall n \in \mathcal{N}, \ \emptyset \subseteq S \subseteq \mathcal{N}.
                \end{equation}
    \item[--]   The transition law $t: Z \times A \times Z \rightarrow [0,1]$ is given by
                \begin{eqnarray}
                \nonumber
                t(z | \emptyset , a_S) := \begin{cases}
                                                1 & \mbox{ for $z = S$}\\
                                                0 & \mbox{ otherwise}
                                            \end{cases} \\
                t(z | \tilde{S}, a_S) := \begin{cases}
                                                1 & \mbox{ for $z = \tilde{S}$}\\
                                                0 & \mbox{ otherwise}
                                             \end{cases}
                \end{eqnarray}
                where $z, S, \tilde{S} \in Z$, $\tilde{S} \neq \emptyset$ and $a_S \in A$.
    \item[--]   The payoff function is given by $\tilde{r} : A \rightarrow
                \mathds{R}^N$, $a_S \mapsto \tilde{r}(a_S) := r_S$.
    \item[--]   There is no discounting in this model.
\end{itemize}
\end{remark}


\textbf{Example 1:}
A typical way to describe two- or three-player quitting games is in a matrix. For example let two players be given. Player one is the so called row player and player two the column player.
\begin{center}
\begin{tabular}{cc|c|c|}\hhline{~~--}
& & \multicolumn{2}{c|}{Player 2} \\
& &  continue & quit \\
\hline
\multicolumn{1}{|c}{Player 1} & \begin{tabular}{c}
             continue \\ quit
           \end{tabular} & \begin{tabular}{c}
                               $\circlearrowleft$ \\
                               $(\ 1 \ ,  -1 \ )$
                            \end{tabular} & \begin{tabular}{c}
                                              $(\ 1 \ ,\ 1 \ )$\\
                                              $( -2 \ , - 2 \ )$
                                            \end{tabular}\\
                                            \hline
\end{tabular}
\end{center}
Where $\circlearrowleft$ means, that the players does not receive any payoff and the game continues to the next round.

In this case the quitting game is given by
\begin{equation*}
G = \left(\{1, 2\}, \left( r_{\emptyset}= \left(\begin{array}{c}
                                            0 \\
                                            0
                                          \end{array}\right)
, r_{\{1\}} = \left(\begin{array}{r}
                                             1 \\
                                             -1
                                           \end{array}\right), r_{\{2\}} = \left(\begin{array}{c}
                                                                      1 \\
                                                                      1
                                                                    \end{array}\right), r_{\mathcal{N}} = \left(\begin{array}{c}
                                                                    -2 \\
                                                                    -2
                                                                    \end{array}\right)\right)\right).
\end{equation*}
\medskip

For further analysis the term ``strategy profile'' is needed.

\begin{definition}[strategy profile, strategy]
Let $G = (\mathcal{N}, (r_S)_{\emptyset \subseteq S \subseteq \mathcal{N}})$ be a given quitting game. A sequence of probability vectors $\pi:=(p_i)_{i \in \mathds{N}}$ with $p_i= (p_i^1, \ldots, p_i^N)^T \in [0,1]^N$ is called strategy profile in the quitting game $G$ for the players $1, \ldots, N$. $p_i^n$ stands for the probability that player $n$ will play the action \emph{quit} at stage $i$. The sequence $\pi^n := (p^n_i)_{i \in \mathds{N}}$ is called strategy for player $n$, $n \in \mathcal{N}$.

Let $\Pi$ be the set of all strategy profiles for the given quitting game.
\end{definition}

\begin{definition}[subgame profile]
Let $G = (\mathcal{N}, (r_S)_{\emptyset \subseteq S \subseteq \mathcal{N}})$ be a given quitting game and $\pi = (p_i)_{i \in \mathds{N}}$ a strategy profile in $G$. For each $j \in \mathds{N}$, $\pi_j := (p_i)_{j \leq i \in \mathds{N}}$ denotes the subgame profile induced by $\pi$ in the quitting game starting at time $j$.
\end{definition}

\begin{definition}[pure, cyclic, stationary]
Let $\pi = (p_i)_{i \in \mathds{N}}$ be a strategy profile in a quitting game $G$. A strategy $\pi^n = (p^n_i)_{i \in \mathds{N}}$ for player $n$ is called
\begin{itemize}
    \item[--]   pure, if $p^n_i \in \{0, 1\}$ for all $i \in \mathds{N}$.
    \item[--]   cyclic, if a $k_0 \in \mathds{N}$ exists such that $p^n_k = p^n_{k + k_0}$ for every $k \in \mathds{N}$.
    \item[--]   stationary, if $p^n_k = p^n_1$ for all $k \in \mathds{N}$.
\end{itemize}
A strategy profile $\pi$ is called pure, if all strategies $\pi^n$, $n \in \mathcal{N}$, are pure. It is cyclic, if all strategies are cyclic, and stationary, if all strategies are stationary.
\end{definition}

\begin{notation}\label{Notation1}
Let $\pi = (p_i)_{i \in \mathds{N}}$ be a strategy profile and $\tilde{\pi}^n = (\tilde{p}^n_i)_{i \in \mathds{N}}$ an alternative strategy for player $n$, $n \in \mathcal{N}$. We denote by $\pi^{-n}$ the strategy profile for the players $j \in \mathcal{N}\setminus \{n\}$ and by $(\pi^{-n}, \tilde{\pi}^n)$ an alternative strategy profile for player $n$ in which the players $j \in \mathcal{N}\setminus \{n\}$ carry on playing $\pi^{-n}$, that means
\begin{equation*}
\pi^{-n} := \left(\begin{array}{llccc}
              p^1_1 & p^1_2 & \ldots  \\
              \vdots & \vdots &   \\
              p^{n-1}_1 & p^{n-1}_2 & \ldots  \\
              p^{n+1}_1 & p^{n+1}_2 & \ldots  \\
              \vdots & \vdots &  \\
              p^N_1 & p^N_2 & \ldots
            \end{array}\right) \quad \mbox{ and } \quad (\pi^{-n}, \tilde{\pi}^n) := \left(\begin{array}{llccc}
              p^1_1 & p^1_2 & \ldots  \\
              \vdots & \vdots &   \\
              p^{n-1}_1 & p^{n-1}_2 & \ldots  \\
              \tilde{p}^n_1 & \tilde{p}^n_2 & \ldots \\
              p^{n+1}_1 & p^{n+1}_2 & \ldots  \\
              \vdots & \vdots &  \\
              p^N_1 & p^N_2 & \ldots
            \end{array}\right).
\end{equation*}
\end{notation}

\subsection{The underlying stochastic process}

Let $G = (\mathcal{N}, (r_S)_{\emptyset \subseteq S \subseteq \mathcal{N}})$ be the given quitting game, $Z = \{S \ | \ \emptyset \subseteq S \subseteq \mathcal{N} \}$ the corresponding state space and $A = \{ 0, 1 \}^N$ the corresponding action space (cf. remark \ref{Bem1}). Set
\begin{equation*}
\Omega := (Z \times A)^{\infty}
\end{equation*}
and
\begin{equation*}
\mathcal{A} := \mathcal{P}(Z) \otimes \mathcal{P}(A) \otimes \mathcal{P} (Z)\otimes \mathcal{P}(A) \otimes \ldots.
\end{equation*}
Furthermore and without loss of generality let $z= \emptyset$ be the initial state. If a strategy profile $\pi = (p_i)_{i \in \mathds{N}} \in \Pi$ is given, a unique probability measure $\mathds{P}_\pi$ on $(\Omega, \mathcal{A})$ and a stochastic process $(X_k, Y_k)_{k \in \mathds{N}}$ with values in $(Z \times A)$ exist, where

\begin{itemize}
    \item[--]   $X_k(\omega) = X_k\big(( z_1, a_1, z_2, a_2, \ldots ) \big) := z_k$

                ($X_k$ denotes the random state of the system at time $k$, $k \in \mathds{N}$, $\omega \in \Omega$),
    \item[--]   $Y_k(\omega) = Y_k\big(( z_1, a_1, z_2, a_2, \ldots ) \big) := a_k$

                ($Y_k$ denotes the random action taken at time $k$, $k \in \mathds{N}$, $\omega \in \Omega$),

    \item[--]   $H_k := (X_1, Y_1, \ldots, X_k)$,

                that means $H_k(\omega) = H_k \big(( z_1, a_1, z_2, a_2, \ldots ) \big) = (z_1, a_1, z_2, a_2, \ldots, z_k)$

                ($H_k$ describes the random history at time $k$, $k \in \mathds{N}$, $\omega \in \Omega$)
\end{itemize}
hold and $\mathds{P}_{\pi}$ is defined by
\begin{itemize}
    \item[--]   $\mathds{P}_{\pi}(X_1 = \emptyset) := 1$,
    \item[--]   $\mathds{P}_{\pi}(X_{k+1} = z| H_k = (z_1, a_1, \ldots, z_k), Y_k = a)
                := t(z | z_k, a)$

                if \ $\mathds{P}_{\pi}( H_k = (z_1, a_1, \ldots, z_k), Y_k = a ) > 0$ and
    \item[--]   $\mathds{P}_{\pi}(Y_k = a| H_k = (z_1, a_1, \ldots, z_k)) :=
                \prod\limits_{\{n \in \mathcal{N}: a^n = 1\}} p_k^n  \prod\limits_{\{m \in \mathcal{N}: a^m = 0\}} (1-p_k^m)$

                if \ $\mathds{P}_{\pi}( H_k = (z_1, a_1, \ldots, z_k)) > 0$
\end{itemize}
with $z, z_i \in Z$ for all $i = 1 \ldots, k$ and $a, a_i \in A$ for all $i = 1, \ldots, k-1$.
\bigskip

Equivalently, $\mathds{P}_{\pi}$ can also be described as the unique probability measure on $(\Omega, \mathcal{A})$ for which
\begin{itemize}
    \item[--]   $\mathds{P}_{\pi}(X_1 = \emptyset) := 1$ and
    \item[--]   $\mathds{P}_{\pi}(H_k = (z_1, a_1, z_2, a_2, \ldots, z_k))$

                $:=
                \mathds{P}_{\pi}(X_1 = z_1)\prod\limits_{i = 1}^{k-1} t(z_{i+1}|z_i, a_i) \cdot \big( \prod\limits_{\{n \in \mathcal{N}: a_i^n = 1\}} p_i^n  \prod\limits_{\{m \in \mathcal{N}: a_i^m = 0\}} (1-p_i^m)\big)$,

                where $z_i \in Z$ for all $i = 1, \ldots, k$ and $a_i \in A$ for all $i = 1, \ldots, k-1$.
\end{itemize}

\subsection{Expected payoffs and equilibria}

In this section the expected payoff for a given quitting game will be defined and additionally the terms \emph{Nash-equilibria}, \emph{$\varepsilon$-equilibria} and \emph{approximate equilibria}.
\medskip

In $\cite{SolVie-QG}$ the expected payoff for a quitting game $G = (\mathcal{N}, (r_S)_{\emptyset \subseteq S \subseteq \mathcal{N}})$ and the strategy profile $\pi \in \Pi$ is defined with a stopping time $\tau: \Omega \rightarrow \mathds{N}\cup\{+\infty\}$ where
\begin{equation*}
\tau(\omega) := \inf\big\{k \in \mathds{N} : Y_k(\omega) \in A\setminus \{(0, \ldots, 0)^T\} \big\}
\end{equation*}
concerning the filtration $(\mathfrak{A}_k)_{k\in \mathcal{N}}$ with $\mathfrak{A}_k := \sigma\{Y_i : 1 \leq i \leq k \}$.

\begin{definition}[Expected payoff]
Let $G = (\mathcal{N}, (r_S)_{\emptyset \subseteq S \subseteq \mathcal{N}})$ be a quitting game and $\pi \in \Pi$ the chosen strategy profile. The expected payoff of the game is given by
\begin{equation*}
\gamma(\pi) := \mathds{E}_{\pi}(\tilde{r}(Y_{\tau})  \mathds{1}_{\{\tau< \infty\}})
\end{equation*}
with $\tilde{r}$ from remark \ref{Bem1}, $\gamma(\pi) = (\gamma^1(\pi), \ldots,  \gamma^N(\pi))^T$ and $\mathds{E}_{\pi}$ as expected value with respect to the probability measure $\mathds{P}_{\pi}$.
\end{definition}



With use of the definition of $\mathds{P}_{\pi}$ and $\tilde{r}(\underline{0}) = r_{\emptyset} = \underline{0}$ one obtains
\begin{eqnarray}
\nonumber \gamma(\pi) & = & \mathds{E}_{\pi}(\tilde{r}(Y_{\tau})  \mathds{1}_{\tau< \infty})\\
\nonumber & = & \sum\limits_{k \in \mathds{N}} \sum\limits_{a_k \in A} \mathds{P}_{\pi} (\tau = k, Y_k = a_k) \cdot \tilde{r}(a_k)\\
\nonumber & = & \sum\limits_{k \in \mathds{N}} \Big( \sum\limits_{a_k \in A} \mathds{P}_{\pi}\big(H_{k-1} = (\emptyset, \underline{0}, \emptyset, \underline{0}, \ldots, \emptyset), Y_k = a_k  \big) \cdot \tilde{r}(a_k) \Big)\\
\nonumber & = & \sum\limits_{k \in \mathds{N}} \Big( \prod\limits_{i=1}^{k-1} \prod\limits_{n \in \mathcal{N}} (1-p_i^n) \cdot \sum\limits_{a_k \in A}  \tilde{r}(a_k) \cdot \prod\limits_{\{n \in \mathcal{N}: a_k^n = 1\}} p_k^n  \prod\limits_{\{m \in \mathcal{N}: a_k^m = 0\}} (1-p_k^m)\Big)
\end{eqnarray}

and with remark \ref{Bem1} follows
\begin{equation*}
\gamma(\pi)
 =  \sum\limits_{k \in \mathds{N}} \Big( \prod\limits_{i=1}^{k-1} \prod\limits_{n \in \mathcal{N}} (1-p_i^n) \cdot \sum\limits_{S \in \mathcal{P}(\mathcal{N})}r_S \cdot \prod\limits_{n \in S} p_k^n  \prod\limits_{m \in \mathcal{N}\setminus S} (1-p_k^m)\Big).
\end{equation*}

\medskip

\begin{definition}\emph{\textbf{($\varepsilon$-equilibrium, Nash-equilibrium, approximate equilibria)}}
Let $G = (\mathcal{N}, (r_S)_{\emptyset \subseteq S \subseteq \mathcal{N}})$ be a quitting game. A strategy profile $\pi = (p_i)_{i \in \mathds{N}}$ is called $\varepsilon$-equilibrium  ($\varepsilon \geq 0$) if for every player $n \in \mathcal{N}$ and every strategy $\tilde{\pi}^n$ of player $n$
\begin{equation}
\gamma^n(\pi) \geq \gamma^n((\pi^{-n},\tilde{\pi}^n)) - \varepsilon
\end{equation}
holds.
\smallskip

The strategy profile $\pi = (p_i)_{i \in \mathds{N}}$ is called Nash-equilibrium or $(0-)$equilibrium if $\pi$ is an $\varepsilon$-equilibrium for $\varepsilon = 0$.
\smallskip

A game has got approximate equilibria, if for all $\varepsilon > 0$ an $\varepsilon$-equilibrium exists.
\end{definition}

\begin{definition}[subgame $\varepsilon$-equilibrium]
Let $G = (\mathcal{N}, (r_S)_{\emptyset \subseteq S \subseteq \mathcal{N}})$ be a quitting game. A strategy profile $\pi = (p_i)_{i \in \mathds{N}}$ is called subgame $\varepsilon$-equilibrium ($\varepsilon \geq 0$) if for all $j \in \mathds{N}$ the subgame profile $\pi_j$ is also an $\varepsilon$-equilibrium in $G$.
\end{definition}
\medskip

\section{One-step game}

The consideration of so called one-step games is an instrument for analyzing quitting games. These games are also known as one-stage games (\cite{Simon}, p. 15) or as one-shot game (\cite{SolVie-QG}, p. 269).

\begin{definition}[One-step game]
Let $G = (\mathcal{N}, (r_S)_{\emptyset \subseteq S \subseteq \mathcal{N}})$ be a given quitting game. For every $v \in \mathds{R}^N$ the tuple
\begin{equation*}
\Gamma_v := (G, v) = \big((\mathcal{N}, (r_S)_{\emptyset \subseteq S \subseteq \mathcal{N}}\big), v)
\end{equation*}
denotes the one-step game corresponding to the quitting game $G$, where the players receive the payoff $v$ if $S = \emptyset$ and $r_S$ otherwise  ($\emptyset \neq S \subseteq \mathcal{N}$).
\end{definition}

A one-step game has only one stage. The transition law, the state and action space are the same as the transition law, the state and action space of the quitting game (cf. remark \ref{Bem1}). For the payoff function $\tilde{r}_v: A \rightarrow \mathds{R}^N$
\begin{equation*}
a \mapsto \tilde{r}_v(a):= \begin{cases}
                                v & \mbox{ if $a = \underline{0}$} \\
                                r_{\{n \in \mathcal{N} | a^n = 1\}} & \mbox{ otherwise.}
                                \end{cases}
\end{equation*}
\smallskip

\begin{definition}[Strategy profile, strategy in the one-step game]
Let $\Gamma_v = (G, v)$ be a given one-step game. A vector $p = (p^1, \ldots, p^N)^T \in [0, 1]^N$ is called strategy profile for the one-step game $\Gamma_v$, where $p^n$ stands for the probability that player $n$ will play the action \emph{quit}. $p^n$ denotes the strategy for player $n$, $n \in \mathcal{N}$, in the one-step game $\Gamma_v$.
\end{definition}

\begin{notation}
Let $p \in [0, 1]´^N$ be a strategy profile for a one-step game $\Gamma_v$ and $\tilde{p}^n$ a strategy for player $n$. Similar to notation \ref{Notation1}, $p^{-n}$ denotes the strategy profile for the players $j \in \mathcal{N} \setminus \{n\}$ and $(p^{-n}, \tilde{p}^n)$ an alternative strategy profile for player $n$ in which the players $j \in \mathcal{N} \setminus \{n\}$ carry on playing $p^{-n}$, that means
\begin{equation*}
p^{-n}:= \left(\begin{array}{l}
                 p^1 \\
                 \vdots \\
                 p^{n-1} \\
                 p^{n+1} \\
                 \vdots \\
                 p^N
               \end{array}\right) \quad \mbox{ and } \quad (p^{-n}, \tilde{p}\,^n):=\left(\begin{array}{l}
                                                                    p^1 \\
                                                                    \vdots \\
                                                                    p^{n-1} \\
                                                                    \tilde{p}\,^n \\
                                                                    p^{n+1} \\
                                                                    \vdots \\
                                                                    p^N
                                                                    \end{array}\right).
\end{equation*}
\end{notation}

Let $\Gamma_v = (G, v)$ and $p \in [0, 1]^N$ be given. Without loss of generality the game starts in the state $z = \emptyset$. The corresponding probability space $(\bar{\Omega}, \bar{\mathcal{A}}, \mathbf{P}_{p})$ and the stochastic process $(\bar{X}_1, \bar{Y}_1, \bar{X}_2)$ are defined by
\begin{itemize}
    \item[--]   $\bar{\Omega} := Z \times A \times Z$,
    \item[--]   $\bar{\mathcal{A}} := \mathcal{P}(Z) \otimes \mathcal{P}(A) \otimes
                \mathcal{P}(Z)$,
                \bigskip

    \item[--]   $\bar{X}_i(\omega) = \bar{X}_i((z_1, a, z_2)) := z_i, \, i = 1, 2$,
    \item[--]   $\bar{Y}_1(\omega) = \bar{Y}_1 ((z_1, a, z_2)) := a$,
                \bigskip

    \item[--]   $\mathbf{P}_{p}(\bar{X}_1 = \emptyset) := 1$,

                $\mathbf{P}_{p}(\{\omega\}) = \mathbf{P}_{p}((z_1, a, z_2))$

                \hspace{3.7em} $:=
                \mathbf{P}_{p}(\bar{X}_1 = z_1) \cdot t(z_2 | z_1, a) \prod\limits_{\{n \in \mathcal{N}: a^n = 1\}} p^n \prod\limits_{\{m \in \mathcal{N}: a^m = 0\}} (1-p^m)$,
\end{itemize}
with $\omega \in \bar{\Omega}$, $z_1, z_2 \in Z$, $a \in A$.
\medskip

The expected payoff $\gamma_v$ for the one-step game $\Gamma_v$ under the strategy profile $p \in [0,1]^N$ is given by
\begin{eqnarray*}
\gamma_v(p) & := & \mathbf{E}_{p} \big(\tilde{r}_v(Y_1)\big)\\
            & =  & \sum\limits_{a_S \in A} \mathbf{P}_{p}(Y_1 = a_S) \cdot \tilde{r}_v(a_S)\\
            & =  & \mathbf{P}_{p}(Y_1 = a_{\emptyset}) \cdot v + \sum\limits_{S \in \mathcal{P}(\mathcal{N})} \mathbf{P}_{p}(Y_1 = a_S) \cdot r_S
\end{eqnarray*}
where $\gamma_v(p) = (\gamma^1_v(p), \ldots, \gamma^N_v(p))^T$ and $\mathbf{E}_p$ is the expected value with respect to the probability measure $\mathbf{P}_p$. $\gamma^n_v(p)$ is the expected payoff for player $n$ ($n \in \mathcal{N}$) in the one-step game $\Gamma_v$ under the strategy profile $p$.
\medskip

\begin{notation}\label{Notation2}
The function $\varrho: [0, 1]^N \times \mathcal{P}(\mathcal{N}) \rightarrow [0, 1]$,
\begin{equation*}
(p, S ) \mapsto \varrho(p, S ) := \prod\limits_{n \in S} p\,^n \prod\limits_{m \in \mathcal{N}\setminus S} (1-p\,^m),
\end{equation*}
with $p = (p\,^1, \ldots, p\,^N)^T$, denotes the probability that a quitting coalition $S$ or -- equivalent to that -- an action $a_S \in A$ is chosen under the vector $p$.
\end{notation}

With this notation for the expected payoff $\gamma_v$ under the strategy profile $p \in [0,1]^N$ follows
\begin{equation*}
\gamma_v(p) = \varrho(p, \emptyset) \cdot v + \sum\limits_{S \in \mathcal{P}(\mathcal{N})} \varrho( p, S) \cdot r_S.
\end{equation*}

\begin{secsatz}\label{Absch-Erw.wert}
Let $\Gamma_v$ be a given one-step game and $p \in [0, 1]^N$ a strategy profile in $\Gamma_v$. Then for the expected payoff $\gamma_v(p)$
\begin{equation*}
\gamma_v^n(p) \in [- \delta_v, \delta_v]
\end{equation*}
holds for all $n \in \mathcal{N}$, where
\begin{equation}\label{delta}
\delta_v := \max \big\{ \max\limits_{n \in \mathcal{N}} |v^n|, \ \max\big\{|r^n_S| \big| \ S \in \mathcal{P}(\mathcal{N})\big\} \big\}.
\end{equation}
\end{secsatz}

\begin{proof}
For all $p \in [0,1]^N$ and all $n \in \mathcal{N}$
\begin{eqnarray*}
\gamma_v^n(p)
& = &   \varrho(p, \emptyset) \cdot v^n + \sum\limits_{S \in \mathcal{P}(\mathcal{N})\setminus \{\emptyset\}} \varrho(p, S) \cdot r^n_S \\
& \leq & \varrho(p, \emptyset) \cdot |v^n| + \sum\limits_{S \in \mathcal{P}(\mathcal{N})\setminus \{\emptyset\}} \varrho(p, S) \cdot | r^n_S |
\end{eqnarray*}
holds. With $\delta_v$ like in (\ref{delta})
\begin{eqnarray*}
\gamma_v^n(p) & \leq & \varrho(p, \emptyset) \cdot \delta_v + \sum\limits_{S \in \mathcal{P}(\mathcal{N})\setminus \{\emptyset\}} \varrho(p, S) \cdot \delta_v\\
& = &   \delta_v \cdot \big(\varrho(p, \emptyset) + \sum\limits_{S \in \mathcal{P}(\mathcal{N})\setminus \{\emptyset\}} \varrho(p, S) \big)
\end{eqnarray*}
follows on the one hand and on the other
\begin{eqnarray*}
\gamma_v^n(p)
& \geq & - \varrho(p, \emptyset) \cdot |v^n| - \sum\limits_{S \in \mathcal{P}(\mathcal{N})\setminus \{\emptyset\}} \varrho(p, S) \cdot | r^n_S |\\
& \geq&   - \delta_v \cdot \big(\varrho(p, \emptyset) + \sum\limits_{S \in \mathcal{P}(\mathcal{N})\setminus \{\emptyset\}} \varrho(p, S) \big).
\end{eqnarray*}
With
\begin{equation*}
\varrho(p, \emptyset) + \sum\limits_{S \in \mathcal{P}(\mathcal{N})\setminus \{\emptyset\}} \varrho(p, S) = 1,
\end{equation*}
$\gamma_v^n(p) \in [-\delta_v, \delta_v]$ holds for all $p \in [0,1]^N$ and $n \in \mathcal{N}$.
\end{proof}

In order to show that the expected payoff $\gamma_v(p)$ is linear in the strategy $p^n$ of player $n$ for all $n \in \mathcal{N}$ the following proposition is needed.

\begin{secsatz}\label{Wahrscheinlichkeit}
For all $p \in [0,1]^N$, all $S \in \mathcal{P}(\mathcal{N})$ and all $i \in \mathcal{N}$
\begin{equation*}
\varrho (p , S) = p^i \cdot \varrho\big((p^{-i},1), S \big) + (1 - p^i) \cdot \varrho\big((p^{-i}, 0), S \big)
\end{equation*}
holds.
\end{secsatz}

\begin{proof}

Case 1: \quad $i \in S$
\smallskip

Because
\begin{equation*}
\varrho(p, S)
 =    \prod\limits_{n \in S} p^n \prod\limits_{n \in \mathcal{N}\setminus S} (1-p^n)
 =    p^i \cdot  \prod\limits_{n \in S \setminus \{i\}} p^n \prod\limits_{n \in
        \mathcal{N}\setminus S} (1-p^n),
\end{equation*}
with $(p^{-i}, 1)^i = 1$ and $p^n = (p^{-i},1)^n$ for all  $n \in \mathcal{N}\setminus \{i\}$, where $(p^{-i},1)^n$ denotes the $n$-th component of the alternative strategy profile for player $i$,
\begin{eqnarray*}
\varrho(p, S)
& = &   p^i \cdot (p^{-i},1)^i \cdot \prod\limits_{n \in S\setminus\{i\}} (p^{-i}, 1)^n
        \prod\limits_{n \in \mathcal{N}\setminus S} \big(1-(p^{-i},1)^n\big)\\
& = &   p^i \cdot \prod\limits_{n \in S} (p^{-i}, 1)^n \prod\limits_{n \in \mathcal{N}\setminus S} \big(1-(p^{-i},1)^n\big)\\
& = &   p^i \cdot \varrho\big((p^{-i}, 1), S\big)
\end{eqnarray*}
follows.

\medskip

Case 2: \quad $i \in \mathcal{N}\setminus S$
\smallskip

Because
\begin{equation*}
\varrho(p, S)
 =   (1-p^i) \cdot \prod\limits_{n \in S} p^n \prod\limits_{n \in
        \mathcal{N}\setminus (S \cup \{i\})} (1-p^n),
\end{equation*}
with $(p^{-i}, 0)^i = 0$ and $p^n = (p^{-i},0)^n$ for all  $n \in \mathcal{N}\setminus \{i\}$
\begin{eqnarray*}
\varrho(p, S)
& = &   (1- p^i) \cdot \big(1-(p^{-i},0)^i\big) \cdot \prod\limits_{n \in S} (p^{-i}, 0)^n \prod\limits_{n \in \mathcal{N}\setminus (S \cup \{i\})} \big(1-(p^{-i},0)^n\big)\\
& = &   (1 - p^i) \cdot \prod\limits_{n \in S} (p^{-i}, 0)^n \prod\limits_{n \in \mathcal{N}\setminus S} \big(1-(p^{-i},0)^n\big)\\
& = &   (1-p^i) \cdot \varrho\big((p^{-i}, 0), S\big)
\end{eqnarray*}
follows.
\medskip

Because of
$\varrho\big((p^{-i},1), S\big) = 0 \mbox{ for $i \in \mathcal{N}\setminus S$}$
and
$\varrho\big((p^{-i},0), S\big) = 0 \mbox{ for $i \in  S$}$,
case 1 and case 2 imply the proposition.

\end{proof}

\begin{secsatz}\label{Zerlegung}
Let $\Gamma_v$ be a given one-step game. Then for all $p \in [0, 1]^N$ and all $n \in \mathcal{N}$
\begin{equation*}
\gamma_v(p) = \gamma_v((p^{-n}, 0))  + p^n \cdot \big( \gamma_v((p^{-n}, 1)) - \gamma_v((p^{-n}, 0)) \big)
\end{equation*}
holds, that means the expected payoff $\gamma_v(p)$ is linear in the strategy $p^n$ of player $n$ for all $p \in [0, 1]^N$ and all $n \in \mathcal{N}$.
\end{secsatz}

\begin{proof}
\begin{equation*}
\gamma_v(p)   =  \varrho(p, \emptyset) \cdot v + \sum\limits_{S \in \mathcal{P}(\mathcal{N})} \varrho( p, S) \cdot r_S
\end{equation*}
With proposition \ref{Wahrscheinlichkeit} and $\varrho\big((p^{-n},1), \emptyset\big) = 0$ one obtains for all $n \in \mathcal{N}$
\begin{eqnarray*}
\gamma_v(p)
& = &   \Big( p^n \cdot \varrho\big((p^{-n},1), \emptyset\big) + ( 1 - p^n)\cdot
        \varrho\big((p^{-n},0), \emptyset\big)\Big) \cdot v \\
&   &   + \sum\limits_{S \in \mathcal{P}(\mathcal{N})} \Big( p^n \cdot \varrho\big((p^{-n},1), S\big) + ( 1 - p^n)\cdot
        \varrho\big((p^{-n},0), S\big)\Big) \cdot r_S\\
& = &   ( 1 - p^n) \cdot \varrho\big((p^{-n},0), \emptyset\big) \cdot v + ( 1 - p^n) \sum\limits_{S \in \mathcal{P}(\mathcal{N})}\varrho\big((p^{-n},0), S\big) \cdot r_S\\
&   &   + \ p^n \sum\limits_{S \in \mathcal{P}(\mathcal{N})} \varrho\big((p^{-n},1), S\big) \cdot r_S \\
\end{eqnarray*}
Furthermore
\begin{equation*}
\gamma_v((p^{-n}, 0))  =  \varrho\big((p^{-n}, 0), \emptyset\big)  \cdot v + \sum\limits_{S \in \mathcal{P}(\mathcal{N})} \varrho((p^{-n},0), S)\cdot r_S
\end{equation*}
and
\begin{equation*}
\gamma_v((p^{-n}, 1))  =  \sum\limits_{S \in \mathcal{P}(\mathcal{N})} \varrho((p^{-n},1), S)\cdot r_S.
\end{equation*}
That implies
\begin{equation*}
\gamma_v(p) = (1- p^n) \cdot \gamma_v((p^{-n}, 0)) + p^n \cdot \gamma_v((p^{-n}, 1)).
\end{equation*}
\end{proof}

\begin{conclusion}\label{Zerlegung2}
Let $p \in [0, 1]^N$ be a strategy profile in the one-step game $\Gamma_v$. The following equations hold:
\begin{enumerate}
	\item	$\gamma_v((p^{-n}, 1)) = p^i \gamma_v\big(\pn{p}{n}{i}{1}{1}\big) + (1- p^i)\gamma_v\big(\pn{p}{n}{i}{1}{0}\big) $
	\item	$\gamma_v((p^{-n}, 0)) = p^i \gamma_v\big(\pn{p}{n}{i}{0}{1}\big) + (1- p^i)\gamma_v\big(\pn{p}{n}{i}{0}{0}\big) $
\end{enumerate}
for all $i, n \in \mathcal{N}$, $i \neq n$.
\end{conclusion}

\begin{definition}\textbf{($\varepsilon$-equilibrium, Nash-equilibrium, approximate equilibria of the}\linebreak \textbf{one-step game)}
Let $\Gamma_v$ be a one-step game corresponding to a quitting game $G$. The strategy profile $p \in [0, 1]^N$ is called an $\varepsilon$-equilibrium for $\varepsilon \geq 0$ if
\begin{equation*}
\forall n \in \mathcal{N} \  \forall \tilde{p}^n \in [0,1] : \gamma^n_v(p) \geq \gamma^n_v \big((p^{-n},\tilde{p}^n)\big) - \varepsilon.
\end{equation*}
If $p$ is an $\varepsilon$-equilibrium with $\varepsilon = 0$, p is also called (Nash-)equilibrium.

A one-step game $\Gamma_v$ has got approximate equilibria, if for all $\varepsilon > 0$ an $\varepsilon$-equilibrium in $\Gamma_v$ exists.
\end{definition}

Because of the linearity of the expected payoff $\gamma_v(p)$ in the strategies $p^n$ ($n \in \mathcal{N}$) it is sufficient to consider the expected payoff only for pure strategies in order to find out whether a given strategy profile in a one-step game is an equilibrium or not, since the extreme values of $\gamma_v(p)$ is for each single player attained in a border point. The following example illustrates this fact.
\bigskip

\textbf{Example 2:} Consider example 1 again. The corresponding one-step game $\Gamma_v$ for a vector $v = (v^1, v^2)^T \in \mathds{R}^2$ is given by
\begin{center}
\begin{tabular}{cc|c|c|}\hhline{~~--}
& & \multicolumn{2}{c|}{Player 2} \\
& &  continue & quit \\
\hline
\multicolumn{1}{|c}{Player 1} & \begin{tabular}{c}
             continue \\ quit
           \end{tabular} & \begin{tabular}{c}
                               $(\ v^1 \ , \ v^2  )$ \\
                               $(\ 1 \ ,  -1 \ )$
                            \end{tabular} & \begin{tabular}{c}
                                              $(\ 1 \ , \ 1 \ )$\\
                                              $(\ -2 \ , -2 \ )$
                                            \end{tabular}\\
                                            \hline
\end{tabular}
\end{center}

Consider four different given $v$'s:
\begin{enumerate}
    \item   $v_1 = \left(\begin{array}{c} 2 \\ 2 \end{array}\right)$:

            The strategy profile $ p = \left(\begin{array}{c} 0.1 \\ 0 \end{array}\right)$ is a 0.1-equilibrium with the expected payoff
            \begin{equation*}
            \gamma_{v_1}(p) = 0.9 \cdot \left(\begin{array}{c} 2 \\ 2 \end{array}\right) + 0.1 \cdot \left(\begin{array}{r} 1 \\ -1 \end{array}\right) = \left(\begin{array}{c} 1.9 \\ 1.7 \end{array}\right).
            \end{equation*}

            Because if player 1 chooses to play \emph{continue} while player 2 keeps on playing \emph{continue} he has got an expected payoff of $\gamma_{v_1}^1((0, 0)^T) = 2$, which is 0.1 better than his expected payoff under $p$. If he plays \emph{quit} his expected payoff would be $\gamma_{v_1}^1((1, 0)^T) = 1$.

            Otherwise
            if player 2 chooses to play \emph{quit} while player 1 keeps on playing \emph{quit} with a probability of 0.1, player 2 gains a payoff of
            \begin{equation*}
            \gamma_{v_1}^2((0.1, 1)^T) = 0.9 \cdot 1 + 0.1 \cdot (-2) = 0.7.
            \end{equation*}
            Player 2 would even change for the worse.
    \item   $v_2 = \left(\begin{array}{c} 0 \\ 2 \end{array}\right)$:

            The strategy profile $ p = \left(\begin{array}{c} 1 \\ 0 \end{array}\right)$ is a Nash-equilibrium with the expected payoff $\gamma_{v_2}(p) = \left(\begin{array}{r} 1 \\ -1 \end{array}\right)$.
            Because if player 1 chooses to play \emph{continue} while player 2 keeps on playing \emph{continue} he has got an expected payoff of $\gamma_{v_2}^1((0, 0)^T) = 0$ and if player 2 chooses to play \emph{quit} while player 1 keeps on playing \emph{quit}, player 2 gains a payoff of $\gamma_{v_2}^2((1, 1)^T) = -2$. Player 2 would even change for the worse, too.
    \item   $v_3 = \left(\begin{array}{c} 2 \\ 0 \end{array}\right)$:

            Analogously to case 2, the strategy profile $ p = \left(\begin{array}{c} 0 \\ 1 \end{array}\right)$ is a Nash-equilibrium with the expected payoff $\gamma_{v_3}(p) = \left(\begin{array}{c} 1 \\ 1 \end{array}\right)$.
    \item   $v_4 = \left(\begin{array}{c} 0 \\ 0 \end{array}\right)$:

            The strategy profile $ p_1 = \left(\begin{array}{c} 1 \\ 0 \end{array}\right)$ is a Nash-equilibrium with the expected payoff \linebreak $\gamma_{v_4}(p_1) = \left(\begin{array}{r} 1 \\ -1 \end{array}\right)$. Analogously $ p_2 = \left(\begin{array}{c} 0 \\ 1 \end{array}\right)$ is also a Nash-equilibrium with the expected payoff $\gamma_{v_4}(p_2) = \left(\begin{array}{c} 1 \\ 1 \end{array}\right)$.
\end{enumerate}
Obviously the choice of $v$ is important. This leads to the question which $v$'s are expedient referring to finding an ($\varepsilon$-)equilibrium in the corresponding quitting game\footnote{Simon therefore introduced in \cite{Simon} the term \emph{feasible} for a vector $v \in \mathds{R}^N$: A vector $v \in \mathds{R}^N$ is feasible if it is in the convex hull of $\{r_S| \emptyset \neq S \subset \mathcal{N} \} \cup \{\underline{0}\}$. } . For example: It does not make sense to choose $v$ like in the first case, because in the corresponding quitting game the expected payoffs are limited by one for each player.
\bigskip

Furthermore proposition \ref{Zerlegung} motivates the definition of the \emph{best reply}, but before stating the definition it is necessary to introduce the mapping $supp$. $supp: [0,1] \rightarrow \mathcal{P}(\{0, 1\})$  denotes the actions that are played with positive probability under $\tilde{p}$, that means
\begin{equation*}
supp(\tilde{p}) := \begin{cases}
\{0\}    & \mbox{ for $\tilde{p} = 0$}\\
\{0, 1\} & \mbox{ for $\tilde{p} \in (0,1)$}\\
\{1\}    & \mbox{ for $\tilde{p} = 1$}
\end{cases}.
\end{equation*}

\begin{definition}[best reply, perfect]
Let $\Gamma_v$ be a given one-step game and $p \in [0,1]^N$ a strategy profile in $\Gamma_v$. An action $b \in \{0, 1\}$ of player $n$ is an $\varepsilon$-best reply for $p^{-n}$ if
\begin{equation*}
\gamma_v^n((p^{-n}, b)) \geq \max\limits_{\tilde{b} \in \{0, 1\}} \gamma_v^n((p^{-n}, \tilde{b})) - \varepsilon
\end{equation*}
$n \in \mathcal{N}$.

A strategy profile $p \in [0,1]^N$ in $\Gamma_v$ is called $\varepsilon$-perfect\,\footnote{Solan an Vieille used in \cite{SolVie-QG} instead of the term ``$\varepsilon$-perfect" the term ``perfect $\varepsilon$-equilibrium". This formulation is confusing with regard to theorem \ref{Theorem-perfect-GW}. The here used phrase is more accurate.}, if for every player $n \in \mathcal{N}$, every action $b \in supp(p^n)$ is an $\varepsilon$-best reply for $p^{-n}$.
\end{definition}

\begin{remark}\label{Bem2}
Let $\Gamma_v$ be the given one-step game and $\varepsilon \geq 0$. The second part of the definition above is equivalent to the following:
\medskip

The strategy profile $p$ for the one-step game $\Gamma_v$ is $\varepsilon$-perfect, if
\begin{equation*}
\forall n \in \mathcal{N} :
\begin{cases}
    \gamma_v^n((p^{-n},1)) - \gamma_v^n((p^{-n},0)) \leq \varepsilon  & \mbox{ for $p^n = 0$}\\
    \gamma_v^n((p^{-n},1)) - \gamma_v^n((p^{-n},0)) \in  [-\varepsilon, \varepsilon]  & \mbox{ for $p^n \in (0, 1)$}\\
    \gamma_v^n((p^{-n},1)) - \gamma_v^n((p^{-n},0)) \geq - \varepsilon  & \mbox{ for $p^n = 1$}
\end{cases}.
\end{equation*}
\end{remark}

Now look at Example 1 again:

\textbf{Example 3:} Consider the one-step game $\Gamma_{v}$ with $v = \left( \begin{array}{c} 0 \\ 2 \end{array} \right)$.
\begin{enumerate}
    \item   $p = \left( \begin{array}{c} 1 \\ 0 \end{array} \right)$:

            $p$ is a Nash-equilibrum in $\Gamma_{v}$, but is $p$ also ($0$-)perfect?

            It holds that
                \begin{equation*}
                p^1 = 1 \ : \quad \gamma_v^1((p^{-1},1)) - \gamma_v^1((p^{-1},0)) = 1 - 0 = 1 \geq 0
                \end{equation*}
                and
                \begin{equation*}
                p^2 = 0 \ :  \quad \gamma_v^2((p^{-2},1)) - \gamma_v^2((p^{-2},0)) = -2 - (-1) = -1 \leq 0.
                \end{equation*}
            So with remark \ref{Bem2} $p$ is perfect.

    \item   $p = \left( \begin{array}{c} 1 \\ 0.1 \end{array} \right)$:

            $p$ is a $0.1$-equilibrium, because:
            \begin{equation*}
            \gamma_v(p) = 0.1 \cdot \left( \begin{array}{c} -2 \\ -2 \end{array} \right) + 0.9 \cdot \left( \begin{array}{r} 1 \\ -1 \end{array} \right) = \left( \begin{array}{r} 0.7 \\ -1.1 \end{array} \right).
            \end{equation*}
            If player one chooses to play \emph{continue}, while player two keeps on playing $p^2$, he gains a payoff of 0.1, which is less than before.

            If player two chooses to play \emph{continue} with certainty, while player one keeps on playing \emph{quit}, he anticipates a payoff of $-1$, which is $0.1$ more than before.

            But is $p$ also $0.1$-perfect?

            The answer is no, because
            \begin{equation*}
            p^2 \in (0.1) \ : \quad \gamma_v^2((p^{-2},1)) - \gamma_v^2((p^{-2},0)) = -2 - (-1) = -1 \notin [-0.1 , \ 0.1].
            \end{equation*}
\end{enumerate}
Which relation exists between ($\varepsilon$-)equilibria strategy profiles and ($\varepsilon$-)perfect strategy profiles ($\varepsilon \geq 0$)?
\medskip

\begin{theorem}\label{Theorem-perfect-GW}
Let $\Gamma_v$ be a given one-step game and $\varepsilon \geq 0$. Then the following propositions hold:
\begin{enumerate}
    \item   $p \in [0,1]^N$ is $\varepsilon$-perfect for $\Gamma_v$ \quad $\Longrightarrow \quad$
            $p$ is an $\varepsilon$-equilibrium in $\Gamma_v$;
    \item   $p \in [0,1]^N$ is an $\varepsilon$-equilibrium in $\Gamma_v$ \quad
            $\Longrightarrow \quad$ $p$ is $\varepsilon \xi_p$-perfect for $\Gamma_v$,

            where
            \begin{equation*}
            \xi_p := \max\limits_{n \in \mathcal{N}} \xi^n_p \quad \textrm{and} \quad \xi^n_p := \begin{cases}
                \max (\frac{1}{p^n}, \frac{1}{1-p^n}) & \mbox{ for $p^n \in (0,1)$} \\
                1 & \mbox{ for $p^n \in \{0, 1 \}$}
                \end{cases}.
            \end{equation*}
\end{enumerate}
\end{theorem}

\begin{proof}
\smallskip

\textbf{1.:} Let $p \in [0,1 ]^N$ be $\varepsilon$-perfect for $\Gamma_v$. It is to show that $p$ is also an $\varepsilon$-equilibrium in $\Gamma_v$, that means
\begin{equation*}
\gamma_v^n(p) \geq \max\limits_{\tilde{p} \in [0,1]} \gamma_v^n \big((p^{-n}, \tilde{p})\big) - \varepsilon
\end{equation*}
for all $n \in \mathcal{N}$.

Because of the linearity of $\gamma_v^n(p)$ with respect to $p^n$ (cf. proposition \ref{Zerlegung}) it is sufficient to show that
\begin{equation}\label{toshow1}
\gamma_v^n(p) \geq \max\limits_{\tilde{p} \in \{0,1\}} \gamma_v^n \big((p^{-n}, \tilde{p})\big) - \varepsilon
\end{equation}
for all $n \in \mathcal{N}$.

Since $p$ is $\varepsilon$-perfect in $\Gamma_v$, the inequality (\ref{toshow1}) follows
immediately for $p^n = 0$ and $p^n =1$.

For $p^n \in (0,1)$ it holds either
\begin{itemize}
    \item[(a)]  $\gamma_v^n\big((p^{-n},1)\big) \geq \gamma_v^n(p) \geq \gamma_v^n\big((p^{-n},0)\big)$ or
    \item[(b)]  $\gamma_v^n\big((p^{-n},0)\big) > \gamma_v^n(p) > \gamma_v^n\big((p^{-n},1)\big)$.
\end{itemize}
Case (a): \quad With $p$ $\varepsilon$-perfect in $\Gamma_v$ and remark \ref{Bem2}, it holds
\begin{equation*}
\gamma_v^n(p) \geq \gamma_v^n\big((p^{-n},0)\big) \geq \gamma_v^n\big((p^{-n},1)\big) - \varepsilon.
\end{equation*}
Because $\gamma_v^n\big((p^{-n},1)\big) = \max\limits_{\tilde{p} \in [0, 1]} \gamma_v^n\big((p^{-n}, \tilde{p})\big)$
\begin{equation*}
\gamma_v^n(p) \geq \max\limits_{\tilde{p} \in [0, 1]} \gamma_v^n\big((p^{-n}, \tilde{p})\big) - \varepsilon
\end{equation*}
follows.
\smallskip

Case (b): \quad Analogously to case (a) with $p$ $\varepsilon$-perfect in $\Gamma_v$ and remark \ref{Bem2}
\begin{equation*}
\gamma_v^n(p) > \gamma_v^n\big((p^{-n},1)\big) \geq \gamma_v^n\big((p^{-n},0)\big) - \varepsilon= \max\limits_{\tilde{p} \in [0, 1]} \gamma_v^n\big((p^{-n}, \tilde{p})\big) - \varepsilon
\end{equation*}
follows.

So for both cases (\ref{toshow1}) holds.
\medskip

\textbf{2.:} Let $p$ be an $\varepsilon$-equilibrium, that means for all $n \in \mathcal{N}$ and for all $\tilde{p} \in [0, 1]$
\begin{equation}\label{toshow2}
\gamma_v^n(p) \geq \gamma_v^n\big((p^{-n}, \tilde{p})\big) - \varepsilon
\end{equation}
holds.
That implies
\begin{equation}\label{shown1}
\gamma_v^n(p) \geq  \gamma_v^n\big((p^{-n}, 1)\big) - \varepsilon
\end{equation}
for $p^n = 0$ and
\begin{equation}\label{shown2}
\gamma_v^n(p) \geq  \gamma_v^n\big((p^{-n}, 0)\big) - \varepsilon
\end{equation}
for $p^n = 1$.

Consider now $p^n \in (0,1)$. For all $\tilde{p} \in [0, 1]$
\begin{equation}\label{Zusatz}
\gamma_v^n(p) = p^n \cdot \gamma_v^n\big((p^{-n},1)\big) + (1 - p^n) \cdot \gamma_v^n\big((p^{-n},0)\big) \geq \gamma_v^n\big((p^{-n},\tilde{p})\big) - \varepsilon.
\end{equation}
holds (c.f. proposition \ref{Zerlegung}).

For $\tilde{p} = 1$, with (\ref{Zusatz})
\begin{equation*}
(1- p^n) \cdot \gamma_v^n\big((p^{-n},0)\big) - (1- p^n) \cdot \gamma_v^n\big((p^{-n},1)\big) \geq - \varepsilon
\end{equation*}
follows and consequently
\begin{equation*}
 \gamma_v^n\big((p^{-n},1)\big) -  \gamma_v^n\big((p^{-n},0)\big) \leq  \frac{\varepsilon}{1- p^n}.
\end{equation*}

For $\tilde{p} = 0$, with (\ref{Zusatz})
\begin{equation*}
p^n \cdot \gamma_v^n\big((p^{-n},1)\big) - p^n \cdot \gamma_v^n\big((p^{-n},0)\big) \geq  - \varepsilon
\end{equation*}
follows and therefore
\begin{equation*}
  \gamma_v^n\big((p^{-n},1)\big) -  \gamma_v^n\big((p^{-n},0)\big) \geq - \frac{\varepsilon}{p^n}.
\end{equation*}
That implies
\begin{equation}\label{shown3}
\gamma_v^n\big((p^{-n},1)\big) -  \gamma_v^n\big((p^{-n},0)\big) \in \left[ -\frac{\varepsilon}{p^n} , \frac{\varepsilon}{1- p^n}\right] \in \left[-\varepsilon \xi_p^n , \varepsilon \xi_p^n \right]
\end{equation}
where
$\xi_p^n :=  \max\left(\frac{1}{1- p^n},\frac{1}{p^n}\right)$.

Denote $M(p) := \{n \in \mathcal{N} \ | \ p^n \in (0, 1) \}$ and
\begin{equation*}
\xi_p := \begin{cases} \max\limits_{n \in M(p)} \xi^n_p & \mbox{ if $M(p) \neq \emptyset$}\\
                        1 & \mbox{ otherwise}
        \end{cases}.
\end{equation*}
With (\ref{shown1}), (\ref{shown2}) and (\ref{shown3})
\begin{equation*}
\forall n \in \mathcal{N} :
\begin{cases}
    \gamma_v^n((p^{-n},1)) - \gamma_v^n((p^{-n},0)) \leq \varepsilon  & \mbox{ for $p^n = 0$}\\
    \gamma_v^n((p^{-n},1)) - \gamma_v^n((p^{-n},0)) \in  [-\varepsilon \xi_p , \varepsilon \xi_p]  & \mbox{ for $p^n \in (0, 1)$}\\
    \gamma_v^n((p^{-n},1)) - \gamma_v^n((p^{-n},0)) \geq -\varepsilon  & \mbox{ for $p^n = 1$}
\end{cases}
\end{equation*}
follows. With remark \ref{Bem2} and $ \xi_p \geq 1$, $p$ is $\varepsilon \xi_p$-perfect in $\Gamma_v$.
\end{proof}

\begin{remark}
With (\ref{shown3}) even
\begin{equation*}
\forall n \in \mathcal{N} :
\begin{cases}
    \gamma_v^n((p^{-n},1)) - \gamma_v^n((p^{-n},0)) \leq \varepsilon  & \mbox{ for $p^n = 0$}\\
    \gamma_v^n((p^{-n},1)) - \gamma_v^n((p^{-n},0)) \in  \left[ -\frac{\varepsilon}{p^n} , \frac{\varepsilon}{1- p^n}\right] & \mbox{ for $p^n \in (0, 1)$}\\
    \gamma_v^n((p^{-n},1)) - \gamma_v^n((p^{-n},0)) \geq -\varepsilon  & \mbox{ for $p^n = 1$}
\end{cases}
\end{equation*}
holds.
\end{remark}

\begin{conclusion}\label{Folgerung-GW}
Let $\Gamma_v$ be a given one-step game.
\smallskip

\quad $p \in [0, 1]^N$ is  $(0$-$)$perfect for $\Gamma_v$ \quad $\Longleftrightarrow$ \quad $p \in [0, 1]^N$ is a Nash-equilibrium in $\Gamma_v$
\end{conclusion}

\begin{conclusion}
Let $\Gamma_v$ be a given one-step game.
\smallskip

\quad $p \in \{0, 1\}^N$ is  $\varepsilon$-perfect for $\Gamma_v$ \quad $\Longleftrightarrow$ \quad $p \in \{0, 1\}^N$ is an $\varepsilon$-equilibrium in $\Gamma_v$

\end{conclusion}

\section{Equilibria in Quitting Games}
This section presents an imported result referring to equilibria in quitting games.
It was proved by Solan and Vieille in \cite{SolVie-QG}, they showed that a
cyclic $\varepsilon$-equilibrium ($\varepsilon > 0$) under some assumptions on the payoff function exists.
\medskip

\subsection{Preview}

This section studies the influence of a variation in one component of the strategy profile $p \in [0, 1]^N$ for a given quitting game $\Gamma_v$. Therefore define $\hat{p} \in [0, 1]^N$ as follows
\begin{equation}\label{def-p-dach}
\hat{p}:= \hat{p}_{m, \lambda}(p) := \big( p^{-m}, (1 -\lambda)p^m + \lambda\big),
\end{equation}
where $p \in [0,1]^N$, $\lambda \in [0,1]$ and $m \in \mathcal{N}$.

That means, $\hat{p}^m$ is a convex combination of $p^m$ and the pure strategy $1$, which accords to the action \emph{quit}. For $\lambda = 0$ one obtains $\hat{p} = p$ and for $\lambda = 1$ $\hat{p} = (p^{-m}, 1)$.
\medskip

\begin{theorem}\label{1}
Let $\Gamma_v$  be a given one-step game, $\lambda \in [0, 1]$, $p \in [0, 1]^N$ and $m \in \mathcal{N}$ an arbitrary but fixed chosen player.
Then the following hold:
\begin{enumerate}
    \item   $\varrho(\hat{p}, \emptyset) = (1 - \lambda)\varrho(p, \emptyset)$

            That means, the probability that all players play \emph{continue} under $\hat{p}$ is for the $\lambda$-fold smaller of the \emph{continue}-probability under $p$.
    \item   $\gamma_v(\hat{p}) = (1- \lambda) \cdot \gamma_v(p) + \lambda \cdot
            \gamma_v\big((p^{-m}, 1)\big)$
    \item   $\| \gamma_v(\hat{p}) - \gamma_v(p) \| \leq
            \lambda\cdot(r_{max} + \delta_v)$

            where $r_{max} := \max \{ |r^n_S| \ \big| \ n \in \mathcal{N}, S \in \mathcal{P}(\mathcal{S})\}$ and \ $\delta_v = \max \{\ \max\limits_{n \in \mathcal{N}} |v^n|, \ r_{max} \}$ \footnote{$\| \cdot \|$ denotes the maximum norm, that means $ \|y \| := \max\limits_{i \in N} |y^i| $ for all $y = (y^1, \ldots, y^N)^T \in \mathds{R}^N$.}
    \item   If $p \in[0,1]^N$ is $\eta$-perfect in $\Gamma_v$ ($\eta \geq 0$) and if $p^m \in (0,1]$ for the given player $m \in \mathcal{N}$ holds,
        then $\hat{p} = \hat{p}_{m, \lambda}$ is $\tilde{\eta}$-perfect in $\Gamma_v$, with $\tilde{\eta} := \max (2\lambda r_{max} + (1-\lambda)\eta, \eta)$.
\end{enumerate}
\end{theorem}

\begin{proof}
\smallskip

\textbf{To 1.:} \quad The definition of $\hat{p}$ (c.f. (\ref{def-p-dach}) ) implies
\begin{eqnarray*}
\varrho(\hat{p}, \emptyset) = \prod\limits_{n \in \mathcal{N}} ( 1- \hat{p}^n) & = & \big( 1- (1- \lambda)p^m - \lambda \big) \cdot \prod\limits_{n \in \mathcal{N}\setminus\{m\}} ( 1- p^n)\\
& = & (1- \lambda) \cdot \prod\limits_{n \in \mathcal{N}} (1 - p^n)\\
& = & ( 1 - \lambda) \cdot \varrho(p, \emptyset).
\end{eqnarray*}
\medskip

\textbf{To 2.:} \quad With proposition \ref{Zerlegung} and the definition of $\hat{p}$
\begin{eqnarray}
\nonumber
\gamma_v(\hat{p})
& = & \hat{p}^m \cdot \gamma_v\big((\hat{p}^{-m}, 1)\big) + (1 - \hat{p}^m) \cdot \gamma_v\big((\hat{p}^{-m}, 0)\big)\\
\nonumber
& = & \big((1- \lambda)p^m + \lambda \big) \cdot \gamma_v\big((p^{-m}, 1)\big) + (1- \lambda)(1 - p^m) \cdot \gamma_v\big((p^{-m}, 0)\big)\\
\nonumber
& = & (1 - \lambda) \Big(p^m \cdot \gamma_v\big((p^{-m}, 1)\big) + (1 - p^m) \cdot \gamma_v\big((p^{-m}, 0)\big)\Big) + \lambda \cdot \gamma_v\big((p^{-m}, 1)\big)\\
& = & (1- \lambda) \cdot \gamma_v(p) + \lambda \cdot \gamma_v\big((p^{-m}, 1)\big) \label{Zerlegung3}
\end{eqnarray}
holds.
\medskip

\textbf{To 3.:} \quad Under use of (\ref{Zerlegung3}) one obtains

\begin{eqnarray*}
\| \gamma_v(\hat{p}) - \gamma_v(p) \| &  = & \big\|(1- \lambda) \cdot \gamma_v(p) + \lambda \cdot \gamma_v\big((p^{-m}, 1)\big) - \gamma_v(p)\big\| \\
& = & \big\| \lambda \cdot \gamma_v\big((p^{-m}, 1)\big) - \lambda \gamma_v(p) \big\|\\
& = & \lambda \cdot \big\| \gamma_v\big((p^{-m}, 1)\big) - \gamma_v(p) \big\|\\
& \leq & \lambda \cdot \big( \big\| \gamma_v\big((p^{-m}, 1)\big) \big\| + \big\| \gamma_v(p) \big\| \big)
\end{eqnarray*}

Because player $m$ plays \emph{quit} with certainty in the alternative strategy profile $(p^{-m},1)$
\begin{equation*}
\gamma_v^n\big((p^{-m}, 1)\big) = \sum\limits_{S \in \mathcal{P}(\mathcal{N})} \varrho\big((p^{-m},1), S\big) \cdot r^n_S  \ \in \ \left[-r_{max}, r_{max}\right]
\end{equation*}
follows for all $n \in \mathcal{N}$ with $r_{max} = \max \{ |r^n_S| \ \big| \ n \in \mathcal{N}, S \in \mathcal{P}(\mathcal{S})\}$.
\begin{equation*}
\Longrightarrow \quad \| \gamma_v(\hat{p}) - \gamma_v(p) \|  \leq  \lambda \cdot \big( r_{max} + \delta_v \big)
\end{equation*}
where $\delta_v = \max \{\ \max\limits_{n \in \mathcal{N}} |v^n|, \ r_{max} \}$.
\medskip

\textbf{To 4.:} \quad For $\lambda = 0$ and as well as for $p^m = 1$, $\hat{p} = p $ follows and therefore $\hat{p}$ is $\eta$-perfect in $\Gamma_v$ in that case.
\medskip

For $\lambda \in (0, 1]$ and $p^m \in (0,1)$ it is to show, that
\begin{equation}\label{perf}
\forall n \in \mathcal{N} :
\begin{cases}
    \gamma_v^n((\hat{p}^{-n},1)) - \gamma_v^n((\hat{p}^{-n},0)) \leq \tilde{\eta}  & \mbox{ for $\hat{p}\,^n = 0$} \\
    \gamma_v^n((\hat{p}^{-n},1)) - \gamma_v^n((\hat{p}^{-n},0)) \in  [-\tilde{\eta}, \tilde{\eta}]  & \mbox{ for $\hat{p}\,^n \in (0, 1)$}  \\
    \gamma_v^n((\hat{p}^{-n},1)) - \gamma_v^n((\hat{p}^{-n},0)) \geq -\tilde{\eta}  & \mbox{ for $\hat{p}\,^n = 1$}
\end{cases}
\end{equation}
holds with $\tilde{\eta} = \max (2\lambda r_{max} + (1-\lambda)\eta, \eta)$.
\medskip

\textbf{Case 1:} \quad Consider player $m$. With $p$ $\eta$-perfect and $p^m \in (0,1)$
\begin{equation*}
\gamma_v\big((\hat{p}^{-m},1)\big) - \gamma_v\big((\hat{p}^{-m},0)\big) = \gamma_v\big((p^{-m},1)\big) - \gamma_v\big((p^{-m},0)\big) \in [- \eta, \eta]
\end{equation*}
follows immediately and therefore the second inequality from (\ref{perf}) for $\hat{p}^m \in (0,1)$ respectively the last  inequality of (\ref{perf}) for $\hat{p} = 1$.
\medskip

\textbf{Case 2:} \quad Consider player $n \in \mathcal{N}\setminus \{m\}$.
\smallskip

With the definition of $\hat{p}$ for all $i \in \mathcal{N}$ and $b \in [0, 1]$
\begin{equation*}
(\hat{p}^{-n}, b)^i = \begin{cases}
                (1- \lambda)\cdot (p^{-n}, b)^i + \lambda & \mbox{ for $i = m$}\\
                (p^{-n}, b)^i & \mbox{ for $i \in \mathcal{N} \setminus \{m\}$}
                \end{cases}
\end{equation*}
follows. Under use of this and equation (\ref{Zerlegung3})
one obtains
\begin{equation*}
\gamma_v^n\big((\hat{p}^{-n},b)\big)
 = (1- \lambda) \cdot  \gamma_v^n\big( (p^{-n},b)\big) + \lambda \cdot \gamma_v^n\big(\pn{p}{n}{m}{b}{1}\big).
\end{equation*}

This implies
\begin{eqnarray}
\nonumber
\gamma_v^n\big((\hat{p}^{-n},1)\big) - \gamma_v^n\big((\hat{p}^{-n},0)\big)
& = & (1- \lambda) \cdot  \gamma_v^n\big( (p^{-n},1)\big) + \lambda \cdot \gamma_v^n\big(\pn{p}{n}{m}{1}{1}\big)\\
\nonumber
&   & - \ (1- \lambda) \cdot  \gamma_v^n\big( (p^{-n},0)\big) - \lambda \cdot \gamma_v^n\big(\pn{p}{n}{m}{0}{1}\big)\\
\nonumber
& = &   (1- \lambda) \cdot \Big( \gamma_v^n\big( (p^{-n},1)\big) -
        \gamma_v^n\big( (p^{-n},0)\big) \Big)\\
&   & +  \ \lambda  \Big( \gamma_v^n\big(\pn{p}{n}{m}{1}{1}\big) -
        \gamma_v^n\big(\pn{p}{n}{m}{0}{1}\big) \Big). \label{Gl3-b}
\end{eqnarray}

\textbf{(a)} \quad Consider player $n \in \mathcal{N}\setminus \{m\}$ with $p^n = 0$.
\begin{equation}\label{Gl2-b}
p^n = 0 \ \wedge \ p \mbox{ $\eta$-perfect} \quad \Longrightarrow \quad \gamma_v^n\big((p^{-n},1)\big) - \gamma_v^n\big((p^{-n},0)\big) \leq \eta
\end{equation}

With use of (\ref{Gl3-b}) and (\ref{Gl2-b})
\begin{eqnarray*}
\gamma_v^n\big((\hat{p}^{-n},1)\big) - \gamma_v^n\big((\hat{p}^{-n},0)\big)
& \leq & (1- \lambda) \cdot \eta\\
&  & + \ \lambda \cdot \Big( \gamma_v^n\big(\pn{p}{n}{m}{1}{1}\big) -
        \gamma_v^n\big(\pn{p}{n}{m}{0}{1}\big)\Big)\\
& \leq & (1- \lambda) \cdot \eta\\
&  & + \ \lambda \cdot \Big( \big|\gamma_v^n\big(\pn{p}{n}{m}{1}{1}\big)\big| +
          \big|\gamma_v^n\big(\pn{p}{n}{m}{0}{1}\big)\big|\Big)\\
& \leq & (1- \lambda) \cdot \eta + \lambda \left( r_{max} + r_{max}\right)\\
& = & 2 \lambda r_{max} + (1- \lambda) \cdot \eta
\end{eqnarray*}
follows.
\medskip

\textbf{(b)} \quad Consider player $n \in \mathcal{N}\setminus \{m\}$ with $p^n \in (0, 1)$.
\begin{equation}\label{Nummer}
p^n \in (0, 1) \ \wedge \ p \mbox{ $\eta$-perfect} \quad \Longrightarrow \quad \gamma_v^n\big((p^{-n},1)\big) - \gamma_v^n\big((p^{-n},0)\big) \in [-\eta, \eta]
\end{equation}

With this analogously to case 2(a) it follows immediately that
\begin{eqnarray*}
\gamma_v^n\big((\hat{p}^{-n},1)\big) - \gamma_v^n\big((\hat{p}^{-n},0)\big)
& \leq & 2 \lambda r_{max} +  (1- \lambda) \cdot \eta \label{Gl5-b}.
\end{eqnarray*}
Under usage of (\ref{Nummer}) and equation (\ref{Gl3-b}) one obtains
\begin{eqnarray}
\nonumber
\gamma_v^n\big((\hat{p}^{-n},1)\big) - \gamma_v^n\big((\hat{p}^{-n},0)\big)
& \geq & -  (1- \lambda) \cdot \eta\\
\nonumber
&  & - \ \lambda \cdot \Big(\big|\gamma_v^n\big(\pn{p}{n}{m}{1}{1}\big)\big| +
          \big|\gamma_v^n\big(\pn{p}{n}{m}{0}{1}\big)\big|\Big)\\
& \geq & - 2 \lambda r_{max} -  (1- \lambda) \cdot \eta
\end{eqnarray}
and therefore
\begin{equation*}
\gamma_v^n\big((\hat{p}^{-n},1)\big) - \gamma_v^n\big((\hat{p}^{-n},0)\big) \in \Big[- \big(2 \lambda r_{max} + (1- \lambda) \cdot \eta\big), \ 2 \lambda r_{max} + (1- \lambda) \cdot \eta\Big].
\end{equation*}
\smallskip

\textbf{(c)} \quad Consider player $n \in \mathcal{N}\setminus \{m\}$ with $p^n = 1$.
\begin{equation}\label{Gl4-b}
p^n = 1 \ \wedge \ p \mbox{ $\eta$-perfect} \quad \Longrightarrow \quad \gamma_v^n\big((p^{-n},1)\big) - \gamma_v^n\big((p^{-n},0)\big) \geq -\eta
\end{equation}
With equation (\ref{Gl4-b}), (\ref{Gl3-b}) and inequality (\ref{Gl5-b})
\begin{equation*}
\gamma_v^n\big((\hat{p}^{-n},1)\big) - \gamma_v^n\big((\hat{p}^{-n},0)\big) \geq - (2 \lambda r_{max} + (1- \lambda)\cdot \eta)
\end{equation*}
follows.

\end{proof}

\pagebreak

\begin{remark}\textbf{(To theorem \ref{1} 3.)}
\smallskip

\begin{enumerate}
\item
Let $\Gamma_v$ be a given one-step game with $v \in [-2r_{max}, 2r_{max}]$, $p \in [0, 1]^N$, where $p^m \in (0,1)^N$ for at least one player $m \in \mathcal{N}$, a strategy profile in $\Gamma_v$ and $\lambda \in (0, 1)$. Solan and Vieille state in \cite{SolVie-QG} the following estimation:
\begin{equation}
\| \gamma_v(\hat{p}) - \gamma_v(p)\| \leq 2  \lambda r_{max}, \label{estimation1}
\end{equation}
with $\hat{p}$ and $r_{max}$ like before.
\medskip

\textbf{Counter-example:} Consider the following one-step game $\Gamma_v$ with $v = \left(\begin{array}{c} 1 \\ 2 \end{array}\right)$.
\begin{center}
\begin{tabular}{cc|c|c|}\hhline{~~--}
& & \multicolumn{2}{c|}{Player 2} \\
& &  continue & quit \\
\hline
\multicolumn{1}{|c}{Player 1} & \begin{tabular}{c}
             continue \\ quit
           \end{tabular} & \begin{tabular}{c}
                               $(\ 1 \ , \ 2  )$ \\
                               $(\ 1 \ ,  -1 \ )$
                            \end{tabular} & \begin{tabular}{c}
                                              $(\ 0 \ , \ 1 \ )$\\
                                              $(-1 \ ,  -0.5 \ )$
                                            \end{tabular}\\
                                            \hline
\end{tabular}
\bigskip
\end{center}

\quad $\Longrightarrow \quad r_{max} =  \max \{|r^n_S| \ \big| \ n \in \mathcal{N}, S \in \mathcal{P}(\mathcal{N}) \} = 1$ \quad and \quad  $\max\limits_{n \in \mathcal{N}} |v^n| = 2$

\quad  $\Longrightarrow \quad \delta_v = 2$

Obviously $p = \left(\begin{array}{c} 0 \\ 0 \end{array}\right)$ is one (and the only) equilibrium in $\Gamma_v$ with the expected payoff \linebreak
$\gamma_v(p) = v$.

Let $\lambda = 0.1$ be given. It holds that $\varrho(p, \emptyset) = 1 > 1- \lambda = 0.9$. Furthermore let $\hat{p}_{\lambda, 1}$ be defined like before, that means
\begin{equation*}
\hat{p}_{\lambda, 1} = \hat{p} = \left(\begin{array}{c} p^1 + \lambda ( 1- p^1) \\ p^2 \end{array}\right)
          = \left(\begin{array}{c} 0 + 0.1 \cdot 1 \\ 0 \end{array}\right)
          = \left(\begin{array}{c} 0.1 \\ 0 \end{array}\right).
\end{equation*}
It holds that
\begin{equation*}
\gamma_v(\hat{p}) = 0.9 \cdot \left(\begin{array}{c} 1 \\ 2 \end{array}\right) + 0.1 \left(\begin{array}{r} 1 \\ -1 \end{array}\right) = \left(\begin{array}{c} 1 \\ 1.7 \end{array}\right).
\end{equation*}
From this
\begin{equation*}
\| \gamma_v(\hat{p}) - \gamma_v(p)\| = \|\left(\begin{array}{c} 1 \\ 1.7 \end{array}\right) -\left(\begin{array}{c} 1 \\ 2 \end{array}\right)\| = 0.3 = \lambda (r_{max} + \delta_v) > 2 \lambda r_{max} = 2 \cdot 0.1 \cdot 1 = 0.2
\end{equation*}
follows. So estimation (\ref{estimation1}) does not hold.

Furthermore the counter-example shows that the estimation in theorem \ref{1} 3. is even the best estimation.

\item Interpretation: Let $\Gamma_v$ be a given one-step game, $\lambda \in [0,1]$ and $p \in [0, 1]^N$ with $p^m \in (0,1)$ for at least one player $m \in \mathcal{N}$. If the \emph{continue} probability of one player $m$ ($m \in \mathcal{N}$) is decreased by the $\lambda$-fold, then the expected payoff of the players changes maximal at $\lambda(r_{max} + \delta_v)$ for a component.
\end{enumerate}
\end{remark}

\begin{remark}[To theorem \ref{1} 4.]\label{Bem-p0}
Assuming that $p^m = 0$ in theorem \ref{1} 4., $\hat{p}^m = \lambda \in (0,1)$ follows. In order to prove that $\hat{p}$ is $\tilde{\eta}$-perfect it is to show that
\begin{equation*}
\gamma_v^m((\hat{p}^{-m},1)) - \gamma_v^m((\hat{p}^{-m},0)) \in  [-\tilde{\eta}, \tilde{\eta}]
\end{equation*}
But with $p$ $\eta$-perfect only
\begin{equation*}
\gamma_v\big((\hat{p}^{-m},1)\big) - \gamma_v\big((\hat{p}^{-m},0)\big) = \gamma_v\big((p^{-m},1)\big) - \gamma_v\big((p^{-m},0)\big) \leq  \eta \leq \tilde{\eta}
\end{equation*}
follows.

For the other direction holds
\begin{eqnarray*}
\gamma_v\big((\hat{p}^{-m},1)\big) - \gamma_v\big((\hat{p}^{-m},0)\big)
& = &  \gamma_v\big((p^{-m},1)\big) - \gamma_v\big((p^{-m},0)\big) \\
& \geq &  - \ \big| \gamma_v\big((p^{-m},1)\big)\big| - \big|\gamma_v\big((p^{-m},0)\big)\big|\\
& \geq & - r_{max} - \delta_v,
\end{eqnarray*}
however this holds for all $p \in [0,1]^N$ in $\Gamma_v$.

But the proof that
\begin{equation*}
\forall n \in \mathcal{N}\setminus \{m\} :
\begin{cases}
    \gamma_v^n((\hat{p}^{-n},1)) - \gamma_v^n((\hat{p}^{-n},0)) \leq \tilde{\eta}  & \mbox{ for $\hat{p}\,^n = 0$}\\
    \gamma_v^n((\hat{p}^{-n},1)) - \gamma_v^n((\hat{p}^{-n},0)) \in  [-\tilde{\eta}, \tilde{\eta}]  & \mbox{ for $\hat{p}\,^n \in (0, 1)$}\\
    \gamma_v^n((\hat{p}^{-n},1)) - \gamma_v^n((\hat{p}^{-n},0)) \geq -\tilde{\eta}  & \mbox{ for $\hat{p}\,^n = 1$}
\end{cases}
\end{equation*}
will remain unaffected from this case.
\end{remark}

\begin{conclusion}
With theorem \ref{1} 4.  it follows immediately that if $p \in [0,1]^N$ with $p^m \in (0, 1)$, for at least one $m \in \mathcal{N}$, is $(0-)$perfect in $\Gamma_v$ (and therefore an equilibrium in $\Gamma_v$), $\hat{p}$ is $2\lambda r_{max}$-perfect in $\Gamma_v$.
\end{conclusion}
\medskip

The following example shows, that the estimation in theorem \ref{1} 4. is even the best approximation.

\textbf{Example 4:} Consider the following one-step game $\Gamma_v$ with $v = (\frac{9}{10}, \frac{10}{9})^T \in \mathds{R}^2$ given by
\begin{center}
\begin{tabular}{cc|c|c|}\hhline{~~--}
& & \multicolumn{2}{c|}{Player 2} \\
& &  continue & quit \\
\hline
\multicolumn{1}{|c}{Player 1} & \begin{tabular}{c}
             continue \\ quit
           \end{tabular} & \begin{tabular}{c}
                               $\left(\begin{array}{cc} 9/10, &  10/9  \end{array}\right)$ \\
                               $\left(\begin{array}{cc} 1, &  -1  \end{array}\right)$
                            \end{tabular} & \begin{tabular}{c}
                                              $\left(\begin{array}{cc} 1/2, &  1  \end{array}\right)$\\
                                              $\left(\begin{array}{cc} -1, &  1  \end{array}\right)$
                                            \end{tabular}\\
                                            \hline
\end{tabular}
\bigskip

$\Longrightarrow \quad r_{max} = 1$

\end{center}
Let $p = \left(\begin{array}{c} 0.1 \\ 0 \end{array}\right)$ be the given strategy profile in $\Gamma_v$. It holds
\begin{eqnarray*}
\gamma_v^1\big((p^{-1}, 1)\big) - \gamma_v^1\big((p^{-1},0)\big) & = & 1 - v_1 = 1 - \frac{9}{10} = 0.1 \in [-0.1, \ 0.1] \ \mbox{ and}\\
\gamma_v^2\big((p^{-2}, 1)\big) - \gamma_v^2\big((p^{-2},0)\big) &  = &  0.1 \cdot 1 + 0.9 \cdot 1 - 0.1 \cdot (-1) - 0.9 \cdot \frac{10}{9} \\
&  = & 0.1 \leq 0.1.
\end{eqnarray*}
Therefore $p$ is $\eta$-perfect in $\Gamma_v$ with $\eta = 0.1$.

Let $\lambda = 0.2$ be given.
\begin{equation*}
\Longrightarrow \quad \hat{p}_{\lambda, 1} = \hat{p} = \left(\begin{array}{c} (1 - \lambda) \cdot p^1 + \lambda \\ p^2 \end{array} \right) = \left(\begin{array}{c} 0.8 \cdot 0.1 + 0.2 \\ 0 \end{array}\right) = \left(\begin{array}{c}  0.28 \\ 0 \end{array}\right)
\end{equation*}
For $\hat{p}$ the following hold
\begin{eqnarray*}
\gamma_v^1\big((\hat{p}^{-1}, 1)\big) - \gamma_v^1\big((\hat{p}^{-1},0)\big) & = & 1 - \frac{9}{10} = 0.1 \in [-0.1, \ 0.1] \ \mbox{ and}\\
\gamma_v^2\big((\hat{p}^{-2}, 1)\big) - \gamma_v^2\big((\hat{p}^{-2},0)\big) &  = &  0.28 \cdot 1 + 0.72 \cdot 1 - 0.28 \cdot (-1) - 0.72 \cdot \frac{10}{9} \\
&  = & 0.48 \\
& = & (1 - \lambda)\eta + 2 \lambda r_{max} = 0.8 \cdot 0.1 + 2 \cdot 0.2 \cdot 1 = 0.48.
\end{eqnarray*}
So $\hat{p}$ is only 0.48-perfect in $\Gamma_v$ and the estimation in theorem \ref{1} 4. holds.

\medskip

\subsection{Equilibria under some assumptions on the payoff function}

This section shows which importance one-step games have, referring to the detection of equilibria in quitting games.

Firstly an important theorem from Solan and Vieille, stated in \cite{SolVie-QG} is quoted. The proof of this theorem is divided into three parts, represented by the propositions \ref{Prop1}, \ref{Prop2} and \ref{Prop3}. Secondly the proposition \ref{Prop1} is proved at length by using the now known results about one-step games and their strategy profiles.

\begin{theorem}\label{Prop0}
Let be $\varepsilon > 0$. Every quitting game $G$ that satisfies the following has a cyclic subgame $\varepsilon$-equilibrium.
\begin{enumerate}
    \item   $r_{\{n\}}^n = 1$ for every $n \in \mathcal{N}$;
    \item   $r_S^n \leq 1$ for every $n \in \mathcal{N}$ and every $S$ such that $n \in S$.
\end{enumerate}
\end{theorem}

Before quoting the above mentioned propositions another notation is needed.

Let $\tilde{V}$ be a subset of $\mathds{R}^N$ and $\varepsilon \in (0, 1)$ be given.
$\psi_{\varepsilon}$ denotes a correspondence\footnote{
Let $K$ and $L$ be sets. A correspondence $J: K \twoheadrightarrow L$ is a subset $J$ of $K \times L$ and one defines for all $k \in K$:
$J(k) := \{ l | (k,l) \in J\}$.
It is not assumed a priori that $J(k) \neq \emptyset$ for all or any particular $k \in K$.} from $\tilde{V}$ into $\tilde{V}$, where
\begin{equation*}
\psi_{\varepsilon}(v) := \psi_{\varepsilon, \tilde{V}}(v) := \big\{ \gamma_v(p) \ \big| \ \gamma_v(p) \in
\tilde{V},
\ p \in [0, 1]^N, \ p \mbox{ $2\varepsilon r_{max}$-perfect}, \ \varrho(p, \emptyset) \leq 1- \varepsilon \big\}.
\end{equation*}


\begin{secsatz}\label{Prop1}
Let $\varepsilon \in (0, 1 )$ be given. Define
\begin{equation*}
V := \big\{ \tilde{v} \in [-2r_{max}, 2r_{max}]^N \ \big| \ \exists n \in \mathcal{N} :
                \tilde{v}^n \leq 1 \big\}.
\end{equation*}
Assume that
\begin{enumerate}
		\item 	$r_{\{n\}}^n = 1$ for every $n \in \mathcal{N}$
		\item	for every $v \in V$ an equilibrium $p$ in $\Gamma_v$ exists, such that either
				\begin{enumerate}
					\item	$p = (0, 0, \ldots, 0)^T$ (that means all players choose \emph{continue}) or
					\item	$p \ne (0, 0, \ldots, 0)^T$ and $\gamma_v^n(p) \leq 1$ hold for some $n \in \mathcal{N}$ with  $p^n > 0$.
				\end{enumerate}
\end{enumerate}
Then  $\psi_{\varepsilon}(v)\neq \emptyset$ for all $v \in V$.
\end{secsatz}

\begin{remark}\label{Bem-p-ex}
\smallskip

\begin{enumerate}
\item
Let $G$ be a quitting game, which satisfies the assumptions of theorem \ref{Prop0}, and $\Gamma_v$ a corresponding one-step game to $G$ with $v \in V$. Then a strategy profile $p \in [0, 1]^N$ exists, which satisfies the assumption 2 of proposition \ref{Prop1}.

\begin{proof}
That every one-step game has got an equilibrium was shown in \cite{Holler}. The proof uses Kakutani's fixed point theorem. Furthermore either $p = (0, \ldots, 0)^T$ or $p \neq (0, \ldots, 0)^T$ holds.
For the second case it is to show, that there a player $m$ with $p^m >0$ and $\gamma_v^m(p) \leq 1$ exists. Because of $p \neq (0, \ldots, 0)^T$ at least one player $m \in \mathcal{N}$ exists with $p^m > 0$.

Consider the case that $p^m = 1$. Then with the assumption 2 of theorem \ref{Prop0}
\begin{eqnarray*}
\gamma_v^m(p) & = &  \sum\limits_{S \in \mathcal{P}(\mathcal{N})} \varrho(p, S) r^m_S \\
& = & \sum\limits_{\emptyset \neq S \subseteq \mathcal{N}\setminus \{m\} } \varrho(p, S\cup\{m\}) r^m_{S \cup \{m\}} \\
& \leq & \sum\limits_{\emptyset \neq S \subseteq \mathcal{N}\setminus \{m\} } \varrho(p, S\cup\{m\}) \cdot 1 \\
& \leq & 1
\end{eqnarray*}
follows.

Consider the case $p^m \in (0, 1)$. Because $p$ is an equilibrium in $\Gamma_v$, $p$ is also (0-)perfect in $\Gamma_v$. This implies $\gamma_v^m(p^{-m}, 1) - \gamma_v^m(p^{-m}, 0)= 0$. Analogously to the case above it follows that $\gamma_v^m(p^{-m}, 1) \leq 1$ and with use of the linearity of $\gamma_v(p)$ in $p^m$ one obtains
\begin{equation}
\gamma_v^m(p) = \gamma_v^m(p^{-m}, 1) \leq 1.
\end{equation}
\end{proof}
\item   The assumptions of proposition \ref{Prop1} are basically there to allow the in remark \ref{Bem-p0} mentioned estimation below for $p = (0, \ldots, 0)$ respectively to ensure that $p^m \in (0, 1]$ in the case $p \neq (0, \ldots, 0)^T$.
\end{enumerate}
\end{remark}

\begin{secsatz}\label{Prop2}
Let $\varepsilon \in (0, 1 )$ be given. If a compact set $V$ exists such that $\psi_{\varepsilon}(v) \neq \emptyset$ for all $v \in V$, then a cyclic profile $\pi = (p_i)_{i \in \mathds{N}}$ in $G$ exists, such that for every $i \in \mathds{N}$:
\begin{enumerate}
	\item	$\pi_i = (p_j)_{i \leq j \in \mathds{N}}$ is terminating\footnote{
Let $G = (\mathcal{N}, (r_S)_{\emptyset \subseteq S \subseteq \mathcal{N}})$ be a given quitting game and $\pi \in \Pi$ the chosen strategy profile. If
$\mathds{P}_{\pi} (\tau < +\infty) = 1$ the game $G$ is called terminating.
} and
	\item	$p_i$ is $(2r_{max} +2) \varepsilon$-perfect in $\gamma_{\gamma(\pi_{i+1})}$.
\end{enumerate}
\end{secsatz}

\begin{secsatz}\label{Prop3}
Let $\pi = (p_i)_{i \in \mathds{N}}$ be a strategy profile in $G$. Assume that the following properties hold for every $i \in \mathds{N}$:
\begin{enumerate}
	\item	$ \pi_i = (p_j)_{i \leq j \in \mathds{N}}$ is terminating and
	\item	$ p_i$ is $\varepsilon$-perfect in $\gamma_{\gamma(\pi_{i+1})}$.
\end{enumerate}
Then either $\pi$ is a subgame $\varepsilon^{\frac{1}{6}}$-equilibrium, or there is a stationary $\varepsilon^{\frac{1}{6}}$-equilibrium.
\end{secsatz}
\medskip

\subsubsection{Proof of Proposition \ref{Prop1}}

\begin{proof}
Let $v \in V$ and $\varepsilon \in (0, 1)$ be arbitrary but fix. The aim is to construct a $\hat{p} \in [0, 1]^N$ with $\gamma_v(\hat{p}) \in \psi_{\varepsilon}(v)$.
\medskip

It holds that $\psi_{\varepsilon}(v) \neq \emptyset$, if a strategy profile $p \in [0, 1]^N$ in $\Gamma_v$ exists, such that
\begin{enumerate}[(i)]
    \item   $\gamma_v(p) \in V = \big\{ \tilde{v} \in [-2r_{max}, 2r_{max}]^N \ \big| \ \exists n
            \in \mathcal{N} : \tilde{v}^n \leq 1 \big\}$
    \item   $p$ is $2\varepsilon r_{max}$-perfect in $\Gamma_v$
    \item   $\varrho(p, \emptyset) \leq 1-\varepsilon$.
\end{enumerate}
Now let $p$ be an equilibrium in $\Gamma_v$ that satisfies the assumptions of the proposition\footnote{Such a probability $p \in [0,1]$ exists, c.f. remark \ref{Bem-p-ex}.}.
\smallskip

\textbf{To   (i):} \hspace{1em} If $p = (0, \ldots, 0)^T$, then $\gamma_v(p) = v \in V$ holds.
In the other case ($p \neq (0, \ldots, 0)^T$) the proposition postulated that a player $m \in \mathcal{N}$ exists such that $\gamma_v^m(p) \leq 1$. Furthermore with the definition of $V$, proposition \ref{Absch-Erw.wert} and
\begin{equation*}
\delta_v = \max \big\{ \max\limits_{n \in \mathcal{N}} |v^n|, \  r_{max} \big\} \leq \max \big\{ \max\limits_{v \in V, n \in \mathcal{N}} |v^n|, \  r_{max} \big\} = 2r_{max}
\end{equation*}
$\gamma_v(p) \in [-2r_{max}, 2r_{max}]^N$ holds.  That implies $\gamma_v(p) \in V$.
\smallskip

\textbf{To  (ii):} \quad Conclusion \ref{Folgerung-GW} implies that $p$ is even 0-perfect in
$\Gamma_v$.
\medskip

\textbf{To (iii):} \hspace{0.3em} For $p = (0, \ldots, 0)^T$, $\varrho(p, \emptyset) = 1 \nleq 1- \varepsilon$ holds and for $p \neq (0, \ldots, 0)^T$, $\varrho(p, \emptyset) < 1$ but not necessarily  $\varrho(p, \emptyset) \leq 1 - \varepsilon$.
\medskip

So $\gamma_v(p)$ is not necessarily an element of $\psi_\varepsilon(v)$.
\bigskip

Based on the given strategy profile $p$, a new profile $\hat{p} \in [0,1]^N$ like in section 3.1 for the one-step game $\Gamma_v$ will be constructed such that $\varrho(\hat{p}, \emptyset) \leq 1 - \varepsilon$ holds. Afterwards it will be shown that this profile $\hat{p}$ satisfies the conditions (i) and (ii), stated at the beginning of this proof, as well.
\medskip

First to the construction of $\hat{p}$: Fix a player $m$ with $v^m =1$ if $p = (0, \ldots, 0)^T$ or with $p^m > 0$ and $\gamma_v^m(p) \leq 1$ otherwise\footnote{Let $p = (0, \ldots, 0)^T$ be the given equilibrium in $\Gamma_v$. Since $v \in V$, a player $m \in \mathcal{N}$ with $v^m \leq 1$ exists. Because $p$ is an equilibrium in $\Gamma_v$, $v^m = 1$ follows. Assume that $v^m < 1$, then player $m$ could change for the better, if he chooses to play \emph{quit} with certainty, hence $r^m_{\{m\}} = 1$ (c.f. assumption 1. of proposition \ref{Prop1}). This is a contradiction to $p$ is an equilibrium.} and set $\hat{p}$ like in (\ref{def-p-dach}), that means
\begin{equation*}
\hat{p}\,^n = \begin{cases}
            (1 - \varepsilon) \cdot p^n + \varepsilon & \mbox{ for $n = m$}\\
            p^n & \mbox{ for $n \neq m$}
            \end{cases}.
\end{equation*}

Theorem \ref{1} 1. implies
\begin{equation*}
\varrho(\hat{p}, \emptyset)  =  (1- \varepsilon) \cdot \varrho(p, \emptyset)
 \leq  1 - \varepsilon.
\end{equation*}

\textbf{To (i):} \hspace{1em} It will be shown that $\gamma_v(\hat{p}) \in V = \big\{ \tilde{v} \in [-2r_{max}, 2r_{max}]^N \ \big| \ \exists n
            \in \mathcal{N} : \tilde{v}^n \leq 1 \big\}$.

With proposition \ref{Absch-Erw.wert} and $\delta_v \leq 2r_{max}$,  $\gamma_v(\hat{p}) \in [- 2r_{max}, 2r_{max}]^N$ follows.
Consider the chosen player $m \in \mathcal{N}$. Because $p$ is an equilibrium in $\Gamma_v$
\begin{equation*}
\gamma_v^m(p) \geq \gamma_v^m\big((p^{-m},\hat{p}\,^m)\big) = \gamma_v^m(\hat{p})
\end{equation*}
holds and with the special choice of player $m$
\begin{equation}\label{toshow-1}
1\geq \gamma_v^m(p) \geq \gamma_v^m(\hat{p})
\end{equation}
follows. So $\gamma_v(\hat{p}) \in V$.
\smallskip

\textbf{To (ii):} \quad It is to show that $\hat{p}$ is $2\varepsilon r_{max}$-perfect in $\Gamma_v$.
\smallskip

Case 1: \quad $p^m \in (0,1]$
\smallskip

With theorem \ref{1} 4. and $p$ $(0-)$perfect in $\Gamma_v$ it follows immediately, that $\hat{p}$ is $2\varepsilon r_{max}$-perfect in $\Gamma_v$.
\medskip

Case 2: \quad $p^m = 0$
\smallskip

(a) \quad Consider player $m$. Because $p$ is an equilibrium in $\Gamma_v$, $p$ is also (0-)perfect in $\Gamma_v$. With this and $p^m = 0$
\begin{eqnarray*}
\gamma_v^m\big((\hat{p}^{-m}, 1)\big) - \gamma_v^m\big((\hat{p}^{-m}, 0)\big) & = & \gamma_v^m\big((p^{-m}, 1)\big) - \gamma_v^m\big((p^{-m}, 0)\big) \\
& = & r^m_{\{m\}} - v^m \\
& = & 1 -1 = 0 \in [-0, +0]
\end{eqnarray*}
follows.

(b) \quad Consider player $n \in \mathcal{N}$. With theorem \ref{1} 4. and remark \ref{Bem-p0}
\begin{equation*}
\forall n \in \mathcal{N}\setminus \{m\} :
\begin{cases}
    \gamma_v^n((\hat{p}^{-n},1)) - \gamma_v^n((\hat{p}^{-n},0)) \leq 2\varepsilon r_{max}  & \mbox{ for $\hat{p}\,^n = 0$}\\
    \gamma_v^n((\hat{p}^{-n},1)) - \gamma_v^n((\hat{p}^{-n},0)) \in  [-2\varepsilon r_{max}, 2\varepsilon r_{max}]  & \mbox{ for $\hat{p}\,^n \in (0, 1)$}\\
    \gamma_v^n((\hat{p}^{-n},1)) - \gamma_v^n((\hat{p}^{-n},0)) \geq -2\varepsilon r_{max}  & \mbox{ for $\hat{p}\,^n = 1$}
\end{cases}
\end{equation*}
holds.

Together with Case 2(a) follows that $\hat{p}$ with $p^m = 0$ is $2 \varepsilon r_{max}$-perfect in the one-step game $\Gamma_v$.
\medskip

So $\hat{p} \in \psi_{\varepsilon}(v) \neq \emptyset$ for all $v \in V$.
\end{proof}





\pagebreak


\begin{thebibliography}{99}
    \bibitem{Flesh}     Flesh, J., Thuijsman, F. and Vrieze, O.J.: \textit{Cyclic
                        Markov Equilibria in a Cubic Game}, International Journal of Game Theory, 26, pp. 303-314, 1997
    \bibitem{Holler}    Holler, M., Illing, G.: \textit{Einf\"uhrung in die Spieltheorie},
                        Springer Verlag, Berlin, 2003
    \bibitem{Nollau}    M\"uller, P. H. and Nollau V.: \textit{Steuerung stochastischer
                        Prozesse}, Akademie-Verlag, Berlin, 1984
    \bibitem{Mojca}     Piskuric, M.: \textit{Vector-valued Markov Games}, Dissertation,
                        TU-Dresden, Fakult\"at Mathematik und Naturwissenschaften, 2001
    \bibitem{Simon}     Simon, R.S.: \textit{The structure on non-zero-sum stochastic games}, Advances
                        in Applied Mathematics 38, pp 1-26, 2007
    \bibitem{Sol}       Solan, E.: \textit{Three-Player Absorbing Games}, Mathematics of Operations
                        Research, 24, pp. 669-698, 1999
    \bibitem{SolVie-QG} Solan, E. and Vieille, N.: \textit{Quitting Games}, Mathematics of Operations
                        Research, 26, pp. 265-285, 2001
    \bibitem{SolVie-Bsp} Solan, E. and Vieille, N.: \textit{Quitting Games - An Example}, International Journal of Game Theory, 31, pp. 365-381, 2002

\end{thebibliography}
\end{document}